\documentclass[10pt] {article}
\usepackage{amsmath}
\usepackage{amssymb}
\usepackage{amsfonts}
\usepackage[margin=1in] {geometry}
\usepackage{theorem}
\numberwithin{equation}{section}

\newcommand{\RR} {\mathbb R}

\newcommand{\beq} {\begin{equation}}
\newcommand{\eeq} {\end{equation}}

\theoremstyle{plain}
\newtheorem{theorem}{Theorem} [section]
\newtheorem{lemma}{Lemma} [section]

\newtheorem{corollary}{Corollary}[section]

\newenvironment{proof}[1][Proof]{\begin{trivlist}
\item[\hskip \labelsep {\bfseries #1}]}{\end{trivlist}}

\title{ {Existence and Scattering for Solutions to Semilinear Wave Equations on High Dimensional Hyperbolic Space}}

\author{Amanda French}

\date{ }
 \begin{document}
\maketitle
\begin{abstract} We prove small-data global existence to semi-linear wave equations on hyperbolic space of dimension $n \geq 3$, for nonlinearities that have the form of a sufficiently high integer power of the solution. We also prove the existence and asymptotic completeness of wave operators in this setting.
\end{abstract}
\section{Introduction} The semiliear wave equation
\beq \label{wave}  \Box u = F(u)
\eeq
with Cauchy data 
\beq u(0, \cdot) = f(x), \partial_t u (0, \cdot) = g(x)
\eeq
where $F(u)$ has the form
\beq F(u) = a |u|^b
\eeq 
has been extensively studied on $\mathbb{R}^{n+1}$. In a number of works including \cite{Georgiev}, \cite{John}, \cite{LS}, \cite{Strauss}, and \cite{Tataru}, it was proved that \eqref{wave} has a small-data global solution when $b$ exceeds $p$  the positive square root of the quadratic
\beq (n-1)p^2 - (n+1)p -2 =0.
\eeq
Recent work has been done in examining this problem on $\mathbb{R} \times M$, where $M$ is hyperbolic space of dimension $n$. In \cite{AP}, Anker and Pierfelice obtain a wider range of dispersive and Strichartz estimates than in the Euclidean case, owing to the better dispersion on hyperbolic space. The resulting global existence results proved first in \cite{MT} for dimension $3$ and then improved and expanded in \cite{AP} are as follows: \\
\\
When $3 \leq n$ and $1 < b < 1 + \frac{4}{n-1}$, \eqref{wave} has a global solution given sufficiently small intial data $(f,g) \in H^{\gamma,2}(M) \oplus H^{\gamma-1,2} (M)$ for $\gamma = \frac{n+1}{4} \frac{b-1}{b+1}$. \\
\\
When $3 \leq n \leq 5$ and $1 + \frac{4}{n-1} \leq b \leq 1+ \frac{4}{n-2}$, \eqref{wave} has a global solution given sufficiently small initial data $(f,g) \in H^{\gamma,2}(M) \oplus H^{\gamma-1,2} (M)$  for $\gamma = \frac{n}{2}-\frac{2}{b-1}$. \\
\\
When $n \geq 6 $ and $1 +\frac{4}{n-1} \leq b \leq \frac{n-1}{2}+\frac{3}{n+1} - \sqrt{(\frac{n-3}{2}+\frac{3}{n+1})^2-4\frac{n-1}{n+2}}$, \eqref{wave} has a global solution given sufficiently small intial data $(f,g) \in H^{\gamma,2}(M) \oplus H^{\gamma-1,2} (M)$ for $\gamma = \frac{n}{2}-\frac{2}{b-1}$. \\
\\
When $n=3$ and $b \geq  5$, \eqref{wave} has a global solution given sufficiently small initial data $(f,g) \in H^{\gamma,2}(M) \oplus H^{\gamma-1,2} (M)$ for $\gamma = \frac{3}{2} -\frac{2}{b-1}$.\\
\\
In this paper we will add to this picture results for large $b$ and large $n$, obtained by using the approach of Lindblad and Sogge in \cite{LS} adapted to this setting. This requires using the Leibniz rule for fractional derivatives, which leads to the additional restriction that $b \in \mathbb{Z}$. We finish by demonstrating the existence and asymptotic completeness of wave operators, allowing us to conclude that the solution obtained scatters to a linear solution over time.

\section{Strichartz Estimates}
We will need to make use of the Strichartz estimates already known in this setting. In all that follows let $M= H^n$ unless otherwise specified.  Set
\beq Tf(t,x) = e^{it\sqrt{-\Delta}} f(x)
\eeq
\beq T^{*}g(x) = \int_{-\infty}^{\infty} e^{-it\sqrt{-\Delta}} g(t,x) dt
\eeq
and a relevant theorem, proved in \cite{MT} and \cite{AP}, is:
\begin{theorem} \label{homogstrichartz} We have the mapping properties $T: H^{\gamma,2}(M) \longrightarrow L^q( \RR, L^p(M))$ and  $T^*:L^{p'}(\RR, L^{q'}(M)) \longrightarrow H^{-\gamma,2}(M)$, whenever $ (p,q, \gamma) \in \mathcal{R} \cup \mathcal{E}$ , where 
\beq \mathcal{R} = \{ (p,q,\gamma) | 2< q < \frac{2(n-1)}{n-3}, 2 \leq p \leq \frac{4q}{(n-1)(q-2)}, \gamma =\frac{1}{2} (n+1)(\frac{1}{2}-\frac{1}{q}) \} \nonumber
\eeq
and
\beq \mathcal{E} = \{ (p,q, \gamma) | \frac{1}{p} \leq \frac{1}{2} (n-1)(\frac{1}{2} - \frac{1}{q}), \gamma = n(\frac{1}{2}-\frac{1}{q}) -\frac{1}{p} \}. \nonumber
\eeq
\end{theorem}
Setting 
\beq \label{V} Vf(t,x) = \int_0^t \frac{\sin(t-s)\sqrt{-\Delta}}{\sqrt{-\Delta}} f(s,x) ds,
\eeq 
so that $Vf=u$ solves the zero-data inhomogeneous equation
\beq \label{inhomoggen} \Box u = f, u(0)=\partial_tu(0)=0
\eeq 
on $\RR \times M$, we also have
\begin{theorem} For $(p,q,\gamma), (\tilde{p}, \tilde{q}, \tilde{\gamma}) \in \mathcal{R} \cup \mathcal{E} $, $(p,\tilde{p}) \neq (2,2)$, we have
\beq V: L^{\tilde{p}'}(\RR, H^{\tilde{\gamma}, \tilde{q}'}(M) ) \longrightarrow L^p(\RR, H^{1-\gamma, q}(M) ).
\eeq 
\end{theorem}
From Theorem 2.2, together with the commutativity of $V$ with $(-\Delta)^{-\frac{\sigma}{2}}$, we also deduce:
\begin{corollary} \label{bigcorollary} In the setting of Theorem 2.2, we have for each $\sigma \in \RR$
\beq ||Vf||_{L^p(\RR, H^{\sigma+1-\gamma,q}(M))} \leq C ||f||_{L^{\tilde{p}'}(\RR, H^{\sigma+\tilde{\gamma}, \tilde{q}'}(M))}
\eeq
\end{corollary}

\section{Existence of Solutions}
Here we will use the theorems of the previous section to prove the following:
\begin{theorem} Assume $M=H^n$, $n \geq 3$, and take $b \in [1+ \frac{4}{n-1}, \infty) \cap \mathbb{Z}$. Then there exists $\epsilon_0 > 0$ such that, if the initial data $(f,g)$ satisfy
\beq ||f||_{H^{\gamma,2}(M)}, ||g||_{H^{\gamma-1,2}(M)} < \epsilon_0,
\eeq 
for 
\beq \label{gammab} \gamma = \frac{n}{2}-\frac{2}{b-1}
\eeq
the equation \eqref{wave} is globally solvable.
\end{theorem}
\begin{proof}
Using the technique of Lindblad and Sogge in \cite{LS}, the method of proof will be Picard iteration on the space
\begin{eqnarray} \label{X} \mathfrak{X}=\{ u \in L^{\frac{2(n+1)}{n-1}} (\RR, H^{\gamma-\frac{1}{2}, \frac{(2(n+1)}{n-1}}(M) )\cap L^q(\RR \times M) : \\
||u||_{L^{\frac{2(n+1)}{n-1}}(\RR, H^{\gamma-\frac{1}{2}, \frac{2(n+1)}{n-1}}(M))}, ||u||_{L^q(\RR \times M)} \leq \delta \}  \nonumber
\end{eqnarray}
with $\gamma$ as in \eqref{gammab} and 
\beq \label{bandq} q = \frac{(n+1)(b-1)}{2}.
\eeq
Note that 
\beq \label{gammain} b \geq 1 + \frac{4}{n-1} \Rightarrow q \geq \frac{2(n+1)}{n-1}, \eeq
 so that in this setting it is possible to have $(q,q, \gamma) \in \mathcal{E}$. Also 
\beq \label{gammaq} \gamma = \frac{n}{2}-\frac{n+1}{q}. 
\eeq
The distance function we put on $\mathfrak{X}$ is:
\beq \label{distance} d(u,v) = || u-v||_{L^{\frac{2(n+1)}{n-1}}(\RR \times M)}.
\eeq
It is an important observation that $\mathfrak{X}$ is complete with respect to this distance. Now, following the standard iteration scheme, we define a sequence $\{ u_i \}$ by setting $u_i$ to solve:
\beq \pm u_{i-1}^b= \Box u_i
\eeq
with
\beq u_i(0,x)=f, \partial_t u_i(0,x)=g
\eeq
and
\beq u_{-1} \equiv 0.
\eeq
Our task is now to demonstrate that the non-linear mapping
\beq \label{picardmap} u_i \rightarrow u_{i+1}
\eeq
is: \\
(i) well-defined \\
(ii) a contraction on $\mathfrak{X}$ under the norm \eqref{distance}. \\
Further, we will need to demonstrate that for $u = \lim_{i\rightarrow \infty} u_i$, we have \\
(iii) $F(u_i) \rightarrow F(u)$ in $D' (\RR \times M)$. \\
To begin, define
\beq \label{sup1} N_i = \sup_{\frac{2(n+1)}{n-1} \leq q \leq \frac{(b-1)(n+1)}{2}} || u_i ||_{L^q(\RR, H^{\frac{n+1}{q}-\frac{2}{b-1}, q}(M))}
\eeq
We pause to note some facts about $N_i$.  First, $N_0$ is finite: Set $u_0(t) = \Xi_0 (f,g)(t)= \cos t \sqrt{-\Delta} f + \frac{\sin t \sqrt{-\Delta}}{\sqrt{-\Delta}}g$, and observe that Thorem \ref{homogstrichartz} and the commutativity of $\Xi_0$ with $(\lambda I-\Delta)^{\frac{\alpha}{2}}$ together imply:
\beq ||u_0||_{L^q(\RR, H^{\frac{n+1}{q}-\frac{2}{b-1}, q}(M))} \lesssim  ||f||_{H^{\gamma,2}(M)} + ||g||_{H^{\gamma-1, 2}(M)}
\eeq
provided $(q,q, \frac{n}{2}-\frac{n+1}{q}) \in \mathcal{E}$ and $\gamma$ is as in \eqref{gammab}.  Thus $N_0$ is finite in this setting, and bounded above by the (small) norm of the initial data. Second, it is also true that, for initial data sufficiently small, we have
\beq \label{Mm} N_m \leq 2 N_0.
\eeq
One proves this by induction on $m$, writing:
\beq u_{i+1} = u_0 + \int_0^t \frac{\sin (t-s) \sqrt{-\Delta}}{\sqrt{-\Delta}} F(u_i)(s) ds,
\eeq
This gives
\beq  \label{m+1} N_{m+1} \leq N_0 + || V(F(u_i)||_{L^{q}(\RR, H^{\frac{n+1}{q}-\frac{2}{b-1}, q}(M))}.
\eeq
 We then use Corollary 2.1 with $\sigma = \frac{n}{2}-\frac{2}{b-1}+1$ to deduce that, as $(q,q, \frac{n}{2}-\frac{n+1}{q})$ and $(\frac{2(n+1)}{n-1}, \frac{2(n+1)}{n-1}, \frac{1}{2})$ are in $\mathcal{E}$,

			\beq  ||V(F(u_i)||_{L^q(\RR, H^{\frac{n+1}{q}-\frac{2}{b-1},q}(M))} \lesssim || u_m^b||_{L^{\frac{2(n+1)}{n+3}}(\RR, H^{\frac{n-1}{2}-\frac{2}{b-1}, \frac{2(n+1)}{n+3}}(M))}
\eeq
and hence
\beq  \label{m+1} N_{m+1} \leq N_0 + || u_m^b||_{L^{\frac{2(n+1)}{n+3}}(\RR, H^{\frac{n-1}{2}-\frac{2}{b-1}, \frac{2(n+1)}{n+3}}(M))}.
\eeq
At this point we will need:
\begin{lemma} For $\sigma \in (0,1)$ and M a Riemannian manifold with $C^{\infty}$ bounded geometry, 
\beq ||uv||_{H^{\sigma,p}(M)} \leq C||u||_{H^{\sigma, s_1}(M)} ||v||_{L^{s_2}(M)} + C||u||_{L^{t_1}(M)} ||v||_{H^{\sigma, t_2}(M))}
\eeq
where $\frac{1}{s_1}+\frac{1}{s_2} = \frac{1}{t_1} +\frac{1}{t_2} = \frac{1}{p}$.
\end{lemma}
We will prove Lemma 3.1 presently, but before that let us see how it implies \eqref{Mm}.  If we apply Lemma 3.1 and the standard Leibniz rule to the last term of \eqref{m+1}, we see that it is bounded by a finite sum of terms that look like:
\beq \label{afterleibniz} \Pi_{j=1}^b ||u_m||_{L^{p_j}(\RR, H^{\alpha_j, p_j}(M))}
\eeq
where
\beq \label{alpha1} 0 \leq \alpha_j \leq \frac{n-1}{2} - \frac{2}{b-1}
\eeq
and
\beq \label{alpha2} \Sigma_{j=1}^b \alpha_j = \frac{n-1}{2} - \frac{2}{b-1}
\eeq
and
\beq \label{ps} \frac{2(n+1)}{n-1} \leq p_j \leq \infty
\eeq
and
\beq \label{adds} \Sigma_{j=1}^b \frac{1}{p_j} = \frac{n+3}{2(n+1)}.
\eeq
Fixing the $\alpha_j$'s to meet the above conditions and considering the definition of $N_m$, we take $p_j$ in \eqref{afterleibniz} to satisfy:
\begin{equation} \label{wantalpha} \frac{n+1}{p_j} - \frac{2}{b-1} = \alpha_j. \end{equation}  Summing over these quantities yields \begin{equation}
\Sigma_{j=1}^{b} \frac{1}{p_j} = \frac{n+3}{2(n+1)},
\end{equation}
and \eqref{alpha2} gives
\beq \frac{2(n+1)}{n-1} \leq p_j \leq \frac{(b-1)(n+1)}{2}.
\eeq
Then for each term in \eqref{afterleibniz}  we have
\begin{equation} ||u_m||_{L^{p_j}(\mathbb{R}, H^{\alpha_j, p_j}(M))} \leq N_m,
\end{equation}
and hence that \eqref{afterleibniz} is bounded above by $N_m^b$. 
Plugging this into \eqref{m+1} gives 
\beq N_{m+1} \leq N_0 + N_m^b,
\eeq
which by induction yields \eqref{Mm} for $N_0$ sufficiently small. Then since $\frac{(n+1)(b-1)}{2}$ and $\frac{2(n+1)}{n-1}$ are in $[\frac{2(n+1)}{n-1}, \frac{(b-1)(n+1)}{2}]$, we see that $||u_m||_{L^q(\RR \times M)}$ and $||u_m||_{L^{\frac{2(n+1)}{n-1}}(\RR, H^{\gamma -\frac{1}{2}, \frac{2(n+1)}{n-1}}(M))} $ are also bounded above by $2N_0$. Hence \eqref{picardmap} is well-defined on $\mathfrak{X}$ for initial data sufficiently small.\\
\\
We must now demonstrate that \eqref{picardmap} is a contraction under the norm \eqref{distance}. Write
\begin{eqnarray} & & \label{contract} ||u_{m+1}-u_{k+1}||_{L^{\frac{2(n+1)}{n-1}}(\RR \times M)} = \\
& &||V(F(u_m)-F(u_k))||_{L^{\frac{2(n+1)}{n-1}}(\RR \times M)} \lesssim \nonumber \\
& &||F(u_m)-F(u_k)||_{L^{\frac{2(n+1)}{n+3}}(\RR \times M)}, \nonumber
\end{eqnarray}
the last line of course coming from Theorem 2.2. Then using
\beq ||u|^b -|v|^b| \lesssim |u-v| (|u|^{b-1}-|v|^{b-1})
\eeq
and
\beq \frac{2}{n+1}+\frac{n-1}{2(n+1)} = \frac{n+3}{2(n+1)},
\eeq
Holder's inequality tells us that the last term in \eqref{contract} is bounded above by
\begin{eqnarray} & & ||u_m-u_k||_{L^{\frac{2(n+1)}{n-1}}(\RR \times M)} (|| |u_m|^{b-1} ||_{L^{\frac{n+1}{2}}(\RR \times M)} + || |u_k|^{b-1} ||_{L^{\frac{n+1}{2}}(\RR \times M)} ) = \nonumber \\
& & ||u_m-u_k||_{L^{\frac{2(n+1)}{n-1}}(\RR \times M)} (||u_m||_{L^q(\RR \times M)}^{b-1} + ||u_k||_{L^q(\RR \times M)}^{b-1}).
\end{eqnarray}
The second term is bounded above by $2 \delta^{b-1}$, giving us the contractivity property. \\
\\
Finally we need to show that $F(u_i) \rightarrow F(u) $ in $D'(\RR \times M)$, where $u$ is the limit of $ \{ u_i \} $ in $\mathfrak{X}$. This step is implicit in our previous arguments:
\begin{eqnarray} & & || F(u_i)- F(u) ||_{L^{\frac{2(n+1)}{n+3}}(\RR \times M)} \lesssim \\
& & ||u_i- u||_{L^{\frac{2(n+1)}{n-1}}(\RR \times M)} (|| |u_i|^{b-1}||_{L^{\frac{n+1}{2}}(\RR \times M)} + || |u|^{b-1}|| _{L^{\frac{n+1}{2}}(\RR \times M)}) \nonumber.
\end{eqnarray}
The second term here is finite given $u, u_i \in \mathfrak{X}$, while the first term goes to zero as $i \rightarrow \infty $.
\end{proof}
We return now to the proof of Lemma 3.1. This result is already established on Euclidean space; see for instance \cite{KP} and \cite{CW}. We will use the Euclidean version in conjunction with:
\begin{lemma}
For $M$ a Riemannian manifold with $C^{\infty}$ bounded geometry, with $s>0$ and $p \in (1, \infty)$,
\beq ||u||^p_{H^{s,p}(M)} \approx  \sum \limits_j ||\phi_j u||^p_{H^{s,p}(M)} + ||u||_{L^p}^p
\eeq
where $\{\phi_j: j \in \mathbb{N} \}$ is a tame partition of unity as defined in (1.27) of \cite{HB}.
\end{lemma}
\begin{proof} The proof of this lemma may be found in Lemma 6.7 of \cite{MT}. The term $C^{\infty}$ \textit{bounded geometry} is defined in (1.19) - (1.23) of \cite{HB} as follows:
First, there exists $R_0 \in \mathbb{R}$ such that for all $p \in M$, the exponential map $\text{Exp}_p: T_p M \longrightarrow M$ maps $B_{R_0}(0)$ to $B_{R_0}(p)$ diffeomorphically. Second, the pull-back of the metric tensor from $B_{R_0}(p)$ to $B_{R_0}(0)$ yields a collection of $n \times n$ matrices $G_p(x)$ such that $ \{ G_p: p \in M \} $ is bounded in $C^{\infty}(B_{R_0}(0), \text{End} (\mathbb{R}^n ))$. Finally, for all $p \in M$, $x \in B_{R_0}(0)$, and $\xi \in \mathbb{R}^n$, we have that $ \xi \cdot G_p(x) \xi \geq \frac{1}{2} |\xi|^2$ and $B_{R_0}(p)$ is geodesically convex. Then, a \textit{tame partition of unity} is one whose supports have a bounded number of overlaps and whose elements $\phi_k$ have the property that $\phi_k \circ \text{Exp}$ is bounded in $C_0^{\infty}$ of a ball in $\mathbb{R}^n$.
\end{proof}
Also of use will be the fact that:
\beq || \phi_j u||_{H^{s,p}(M)} \approx || \phi_j u  \circ Exp ||_{H^{s,p}(\RR^n)}.
\eeq
We then write:
\begin{eqnarray} & &|| u v||_{H^{\sigma,p}(M)} = (||u v||_{H^{\sigma, p}(M)}^p)^{\frac{1}{p}} \\
&\approx& (\sum \limits_j  ||\phi_j (uv) ||^p_{H^{\sigma, p}(M)} + ||uv||_{L^p(M)}^p) ^{\frac{1}{p}} \nonumber \\
&\leq& C( \sum \limits_j  ||\phi_j uv \circ Exp ||_{H^{\sigma, p}(\RR^n)}^p)^{\frac{1}{p}} + ||uv||_{L^p(M)} \nonumber \\
&\leq& C(\sum \limits_j  ( ||\phi_j u||_{H^{\sigma,s_1}(M)} ||\phi_j v ||_{L^{s_2}(M)}+ ||\phi_j u ||_{L^{t_1}(M)} || \phi_j v ||_{H^{\sigma, t_2}(M)} )^p )^{\frac{1}{p}} +||uv||_{L^p(M)} \nonumber \\
&\leq& C( \sum \limits_j  (||\phi_j u||_{H^{\sigma,s_1}(M)} ||\phi_j v ||_{L^{s_2}(M)})^p)^{\frac{1}{p}} \nonumber \\
&+& C( \sum \limits_j  ( ||\phi_j u ||_{L^{t_1}(M)} || \phi_j v ||_{H^{\sigma, t_2}(M)} )^p )^{\frac{1}{p}} +||uv||_{L^p(M)}\nonumber \\
&\leq& C(\sum \limits_j   (||\phi_j u||_{H^{\sigma,s_1}(M)})^{s_1})^{\frac{1}{s_1}} (\sum \limits_j  ( ||\phi_j v ||_{L^{s_2}(M)})^{s_2})^{\frac{1}{s_2}}  \nonumber \\ 
&+& C(\sum \limits_j  ( ||\phi_j u ||_{L^{t_1}(M)})^{t_1})^{\frac{1}{t_1}} (\sum \limits_j  ( || \phi_j v ||_{H^{\sigma, t_2}(M)} )^{t_2} )^{\frac{1}{t_2}} + ||uv||_{L^p(M)} \nonumber \\
&\leq& C ||u||_{H^{\sigma, s_1}(M)} ||v||_{L^{ s_2}(M)} + C ||u||_{L^{t_1}(M)} ||v||_{H^{\sigma, t_2}(M)} +  ||uv||_{L^p(M)}. \nonumber
\end{eqnarray}
The last term will be dealt with via Holder's inequality, to write
\beq ||uv||_{L^p(M)} \leq ||u||_{L^{s_1}} ||v||_{L^{s_2}}.
\eeq
This concludes the proof of Lemma 3.1.
\section{Scattering}
 In this section we examine the asymptotic behavior of the solution $u$ to \eqref{wave} in the setting of Theorem 3.1. First, we define:
\beq w= \begin{pmatrix} u \\ u_t \end{pmatrix},\text{     }h= \begin{pmatrix} f \\ g \end{pmatrix}, \text{     } G(w) = \begin{pmatrix}0 \\ F(u) \end{pmatrix}, \text{     }iL= \begin{pmatrix} 0 & I \\ \Delta & 0 \end{pmatrix},
\eeq
so that \eqref{wave} may then be rewritten:
\beq w(t) = e^{itL}h + \int_0^te^{i(t-s)L}G(w(s))ds,
\eeq
or
\beq \label{scatter} e^{-itL}w(t)=h+\int_0^te^{-isL}G(w(s))ds,
\eeq
where $e^{itL} = \begin{pmatrix} \cos{tA} & A^{-1}\sin{tA} \\ -A\sin{tA} & \cos{tA} \end{pmatrix} $, $A = \sqrt{-\Delta}$. \\
We will investigate the convergence of \eqref{scatter} as $t\rightarrow +\infty$ and $t \rightarrow -\infty$.
 \eqref{scatter} implies:
\begin{eqnarray} & & e^{-it_2L}w(t_2) - e^{-it_1L}w(t_1) \nonumber \\ &=& \int_{t_1}^{t_2}e^{-isL}G(w(s))ds \nonumber \\ &=& \int_{t_1}^{t_2} \begin{pmatrix}-A^{-1}\sin{(sA)} F(u(s)) \\ \cos{(sA)} F(u(s)) \end{pmatrix}ds \nonumber \\ &=& 
\begin{pmatrix} \phi_{t_1 t_2} \\ \psi_{t_1 t_2} \end{pmatrix}. \nonumber
\end{eqnarray}
Now set
\beq H_{t_1t_2}(s)=F(s)\chi_{[t_1,t_2]}(s), \text{     }F(s)=F(s,x). \nonumber
\eeq
Then
\begin{eqnarray} \label{H} \int_{t_1}^{t_2}e^{-isA}F(s)ds &=& \int_{-\infty}^{\infty}e^{-isA}H_{t_1t_2}(s)ds \nonumber \\ &=& T^*H_{t_1t_2} \nonumber
\end{eqnarray}
for $T^*$ as in Section 2. We note that $T^*$ commutes with powers of $A$, and this together with Theorem 2.1 yields
\beq \label{T*scattermap} T^*:L^{\tilde{p}'}(\RR, H^{\sigma, \tilde{q}'}(M)) \longrightarrow H^{\sigma-\tilde{\gamma}, 2}(M)
\eeq
for all $\sigma \in \RR$ and $(\tilde{p}, \tilde{q}, \tilde{\gamma}) \in \mathcal{R} \cup \mathcal{E}$. Taking $(\frac{2(n+1)}{n-1}, \frac{2(n+1)}{n-1}, \frac{1}{2}) \in \mathcal{E}$ and $\sigma = \gamma -\frac{1}{2}$, we obtain
\begin{equation} \label{scattermap2} T^*: L^{\frac{2(n+1)}{n+3}}(\RR, H^{\gamma-\frac{1}{2}, \frac{2(n+1)}{n+3}}(M)) \longrightarrow H^{\gamma-1,2}(M).
\end{equation}
This yields 
\begin{eqnarray} & &\Bigl \| \int_{t_1}^{t_2} e^{-isA}F(s)ds \Bigr \| _{H^{\gamma-1,2}(M)} \\ &\lesssim& 
|| T^* H_{t_1 t_2} ||_{H^{\gamma-1,2}(M)} \nonumber \\  &\lesssim&
||F||_{L^{\frac{2(n+1)}{n+3}}([t_1, t_2], H^{\gamma-\frac{1}{2}, \frac{2(n+1)}{n+3}}(M))} \nonumber.
\end{eqnarray}
From Section 3 we know that the right hand side is bounded above by $2 N_0$ which is in turn bounded above by the small norm of the initial data. Hence, we may say that 
\beq || \phi_{t_1 t_2} ||_{H^{\gamma, 2}(M)}, || \psi_{t_1 t_2} ||_{H^{\gamma-1, 2}(M)} \longrightarrow 0
\eeq
as $t_1, t_2 \longrightarrow \pm \infty$. Thus 
\beq \label{Cauchy} e^{-itL}w(t) \text{ is Cauchy in } H^{\gamma, 2}(M) \oplus H^{\gamma-1,2}(M)\\
 \text{ as either } t \rightarrow \infty \text{ or } t\rightarrow -\infty.
\eeq
From \eqref{Cauchy} and the fact that $\{ e^{itL}: t \in \RR \}$ is a uniformly bounded family of operators on $ H^{\gamma, 2}(M) \oplus H^{\gamma-1,2}(M)$, we have the following scattering result:
\begin{theorem}In the setting of Theorem 3.1, with $\gamma$ as in \eqref{gammab}, $(f,g) \in H^{\gamma,2}(M) \oplus H^{\gamma-1,2}(M)$ with sufficiently small norm, and $u$ the solution to \eqref{wave}, there exist
\beq (\phi_{\pm}, \psi_{\pm}) \in H^{\gamma,2}(M) \oplus H^{\gamma-1,2}(M)
\eeq such that 
\beq \Bigl \| \begin{pmatrix} u(t) \\ u_t (t) \end{pmatrix} - e^{itL} \begin{pmatrix} \phi_{\pm} \\ \psi_{\pm} \end{pmatrix} \Bigr \|_{H^{\gamma,2}(M) \oplus H^{\gamma-1,2}(M)} \longrightarrow 0 \text{ as } t \longrightarrow \pm \infty.
\eeq
\end{theorem}
\section{Wave Operators}
Having analyzed the asymptotic behavior of solutions to \eqref{wave} as $t\longrightarrow \pm \infty$, we will now define wave operators and prove their existence in this context.  \\
\\From the previous section, we know that, given the Cauchy problem \eqref{wave}, it is possible to find initial data $\begin{pmatrix} \phi_{\pm} \\ \psi_{\pm} \end{pmatrix}$ that, when acted upon by the linear operator 
\beq \label{Sn} S_n(t) = e^{itL}= \begin{pmatrix} \cos{tA} & A^{-1}\sin{tA} \\ -A\sin{tA} & \cos{tA} \end{pmatrix}, \nonumber \eeq 
\beq A = \sqrt{-\Delta}, \nonumber
\eeq
yields a solution asymptotically close to that of \eqref{wave} as $t \longrightarrow \pm \infty$.  Now we posit an inverse problem:  Given $\begin{pmatrix} \phi_{\pm} \\  \psi_{\pm} \end{pmatrix}$ as initial data, is it possible to obtain a solution to \eqref{wave}?  In other words, we ask if there exist well-defined operators
\beq \label{woperator-} W_-: \begin{pmatrix} \phi_- \\ \psi_- \end{pmatrix} \longrightarrow \begin{pmatrix} u \\ u_t \end{pmatrix} 
\eeq 
and
\beq \label{woperator+} W_+: \begin{pmatrix} \phi_+ \\ \psi_+ \end{pmatrix} \longrightarrow \begin{pmatrix} u\\ u_t \end{pmatrix}.
\eeq
If $W_-$ and $W_+$ exist, we call them wave operators.  It turns out that in this context we can indeed find wave operators, provided $\begin{pmatrix} \phi_{\pm} \\ \psi_{\pm} \end{pmatrix}$ lie in the space $H^{\gamma,2} \oplus H^{\gamma-1,2}$ and have sufficiently small norm. The relevant theorem is as follows:
\begin{theorem} In the setting of Theorem 3.1, there exists an $\epsilon_0$ with the following property:  For $\phi_- \in H^{\gamma,2}(M)$ and $\psi_- \in H^{\gamma-1,2}(M)$ with 
\beq ||\phi||_{H^{\gamma,2}(M)}, ||\psi||_{H^{\gamma-1,2}(M)} \leq \epsilon_0
\eeq
the equation
\beq \label{waveop} w(t) = e^{itL} \begin{pmatrix} \phi_- \\ \psi_- \end{pmatrix}+ \int_{-\infty}^t e^{i(t-s)L} G(w(s)) ds
\eeq
has global solution, satisfying $w = (u, \partial_t u)$, with
\beq u \in L^{\frac{2(n+1)}{n-1}}(\RR, H^{\gamma-\frac{1}{2}, \frac{2(n+1)}{n-1}}(M)) \cap L^q(\RR \times M)
\eeq
where $q = \frac{(n+1)(b-1)}{2}$.
\end{theorem}
\begin{proof} Solving \eqref{waveop} is equivalent to solving
\beq u(t)=(\cos t A) \phi_- + A^{-1} (\sin t A) \psi_- + \int_{-\infty}^t A^{-1} \sin (t-s) A F(u(s)) ds.
\eeq
As before we can find a solution via an interation argument on the space $\mathfrak{X}$ in \eqref{X}, making use of the Leibniz rule for fractional derivatives and the Strichartz estimates of section 2. The only difference is that here we have 
\beq \int_{-\infty}^t A^{-1} \sin (t-s) A F(u(s)) ds = VF (v) (t)
\eeq
where this $V$ is like the $V$ in \eqref{V}, but with $\int_0^t$ replaced by $\int_{-\infty}^t $. The proof of Theorem 2.2 may be trivially exended to include this case, giving the desired result. 
\end{proof}
Having $ w = (u, u_t)$, we now estimate the difference:
\begin{equation} \label{diff}  e^{-itL}w(t)  - \begin{pmatrix} \phi_- \\ \psi_- \end{pmatrix} = \int_{-\infty}^t \begin{pmatrix} -A^{-1} \sin sA F(u(s)) \\ \cos s A F(u(s)) \end{pmatrix} ds.
\end{equation}
Parallel to \eqref{H}, we have, for all real $\sigma$, 
\beq \Bigl \| \int_{-\infty}^t e^{isA}F(s) ds \Bigr \|_{H^{\sigma, 2}} = || T^* H_t ||_{H^{\sigma, 2}}
\eeq
where
\beq H_t (s,x) = \chi_{(-\infty,t]}(s) F(s,x).
\eeq
Setting $\sigma = \gamma -1$ and noting again that $(\frac{2(n+1)}{n-1}, \frac{2(n+1)}{n-1}, \frac{1}{2}) \in \mathcal{E}$, we apply \eqref{T*scattermap} to obtain
\begin{eqnarray} \Bigl \| \int_{-\infty}^t e^{-isA}F(s) ds \Bigr \|_{H^{\gamma-1,2}} \lesssim ||F(u) ||_{L^{\frac{2(n+1)}{n+3}}( (-\infty, t], H^{\gamma-\frac{1}{2}, \frac{2(n+1)}{n+3}}(M))}.
\end{eqnarray}
The right-hand side here may be bounded above (using Lemma 3.1) by the norm of the initial data. We may then write the right-hand side of \eqref{diff} as
\beq \begin{pmatrix} \phi(t) \\ \psi(t) \end{pmatrix},
\eeq
and we have
\beq ||\phi(t)||_{H^{\gamma,2}(M)} + ||\psi(t)||_{H^{\gamma-1,2}(M)} \lesssim ||F(u) ||_{L^{\frac{2(n+1)}{n+3}}( (-\infty, t], H^{\gamma-\frac{1}{2}, \frac{2(n+1)}{n+3}}(M))} \longrightarrow 0,
\eeq
as $t \rightarrow -\infty$. Hence we have the conclusion:
\begin{theorem} In the setting of Theorem 5.1, there exists $\epsilon_0 >0$ such that if $\phi_-$ and $\psi_-$ are chosen satisfying $|| \phi_-||_{H^{\gamma,2}(M)} \leq \epsilon_0$ and $||\psi_-||_{H^{\gamma-1,2}(M)} \leq \epsilon_0 $, then \eqref{waveop} has a solution $w = (u, u_t)$, with $u \in L^{\frac{2(n+1)}{n-1}}(\RR, H^{\gamma-\frac{1}{2}, \frac{2(n+1)}{n-1}}(M)) \cap L^q(\RR \times M)$, with $\gamma = \frac{n}{2}-\frac{2}{b-1}$ and $q = \frac{(n+1)(b-1)}{2}$, and
\beq \Bigl \| \begin{pmatrix} u(t) \\ u_t (t) \end{pmatrix} - e^{itL}\begin{pmatrix} \phi_- \\ \psi_-\end{pmatrix} \Bigr \|_{H^{\gamma,2}(M) \oplus H^{\gamma-1,2}(M)} \longrightarrow 0 \text{ as } t \rightarrow -\infty.
\eeq
\end{theorem}
We can, of course, obtain a similar result for $t \rightarrow \infty$, through a trivial modification of the preceding arguments.

\end{document}